\documentclass[a4paper,10pt]{article}
\usepackage{amsmath,amsthm,xypic,
amssymb,
pdfsync
}
\newtheorem{thm}{Theorem}[section]

\newtheorem{prop}[thm]{Proposition}
\newtheorem{cor}[thm]{Corollary}

\newtheorem{lem}[thm]{Lemma}
\theoremstyle{definition}

\newtheorem{defi}[thm]{Definition}

\newtheorem{rem}[thm]{Remark}

\newtheorem{notation}[thm]{Notation}

\newtheorem{ex}[thm]{Example}

\theoremstyle{plain}

\newcommand\iso{\xrightarrow{
   \,\smash{\raisebox{-0.3ex}{\ensuremath{\scriptstyle\sim}}}\,}}
\newcommand{\lla}{\langle\!\langle}
\newcommand{\rra}{\rangle\!\rangle}

\renewcommand{\phi}{\varphi}

\newcommand{\alt}{\mathop{\mathrm{Alt}}}

\newcommand{\nrd}{\mathop{\mathrm{Nrd}}}
\newcommand{\trd}{\mathop{\mathrm{Trd}}}

\newcommand{\car}{\mathop{\mathrm{char}}}
\newcommand{\id}{\mathop{\mathrm{id}}}

\newcommand{\disc}{\mathop{\mathrm{disc}}}

\author{A.-H. Nokhodkar}
\title{On biquaternion algebras with orthogonal involution}
\date{}
\begin{document}
\maketitle
\begin{abstract}
We investigate the pfaffians of decomposable biquaternion algebras with involution of orthogonal type.
In characteristic two, a classification of these algebras in terms of their pfaffians and some other related invariants is studied.
Also, in arbitrary characteristic, a criterion is obtained for an orthogonal involution on a biquaternion algebra to be metabolic.
\end{abstract}

\section{Introduction}
A biquaternion algebra is a tensor product of two quaternion algebras.
Every biquaternion algebra is a central simple algebra of degree $4$ and exponent $2$ or $1$.
A result proved by A. A. Albert shows that the converse is also true (see \cite[(16.1)]{knus}).
An {\it Albert form} of a biquaternion algebra $A$ is a $6$-dimensional quadratic form with trivial discriminant whose Clifford algebra is isomorphic to $M_2(A)$.
According to \cite[(16.3)]{knus}, two biquaternion algebras over a field $F$ are isomorphic if and only if their Albert forms are similar.

The Albert form of a biquaternion algebra with involution arises naturally~as the quadratic form induced by a {\it pfaffian} (see \cite[(3.3)]{knus5}).
In \cite{knus5}, a pfaffian of certain modules over Azumaya algebras was defined and used to find a decomposition criterion for involutions on a rank $16$ Azumaya algebra, which contains $2$ as a unit.
A similar criterion for involutions on a biquaternion algebra in arbitrary characteristic was also obtained in \cite{parimala}.

It is known that symplectic involutions on a biquaternion algebra $A$ can be classified, up to conjugation, by their {\it pfaffian norms} (see \cite[(16.19)]{knus}).
For orthogonal involutions the situation is a little more complicated.
In characteristic $\neq2$, using \cite[(5.3)]{knus5} one can find a classification of decomposable orthogonal involutions on $A$ in terms of the pfaffian and the {\it pfaffian adjoint} (introduced in \cite{knus5}).
This classification was originally stated in \cite{knus5} for the more general case where $A$ is an Azumaya algebra which contains $2$ as a unit.

In this work we study decomposable biquaternion algebras with ortho\-gonal involution.
We start with some general observations on the pfaffian and the pfaffian adjoint.
For a decomposable orthogonal involution $\sigma$ we consider the pfaffian $q_\sigma$ and certain subsets $\alt(A,\sigma)^+$ and $\alt(A,\sigma)^-$ of $\alt(A,\sigma)$, introduced in \cite{parimala}.
It is shown in (\ref{x2}) that the union of $\alt(A,\sigma)^+$ and
$\alt(A,\sigma)^-$ coincides with the set of all square-central elements in $\alt(A,\sigma)$.
At the end of \S\ref{sec-pfaffian}, we study in more detail the classification of orthogonal involutions on biquaternion algebras in characteristic $\neq2$, obtained in \cite{knus5}.
Although this result is already presented in \cite{knus5}, it is useful to rephrase it to enable comparison with the corresponding result in characteristic $2$ (see (\ref{charn}) and (\ref{char})).

The classification problem in characteristic $2$ is a little more complicated.
Moreover, the results
themselves have some substantial differences in this case.
For example, the restriction $q_\sigma^+$ of $q_\sigma$ to $\alt(A,\sigma)^+$ is totally singular in characteristic $2$, rather than a regular subform of the pfaffian $q_\sigma$.
Considering these remarks, our approach is to study the relation between the form $q_\sigma^+$ and  the {\it Pfister invariant} of $(A,\sigma)$, introduced in \cite{dolphin3}.
This  relation is used  in (\ref{char}) to obtain necessary and sufficient conditions for orthogonal involutions to be conjugate to each other.

Finally, we study in \S\ref{sec-met} metabolic involutions on biquaternion algebras.
Using some results of previous sections,
we obtain various criteria for an orthogonal involution on a biquaternion algebra to be metabolic (see (\ref{hyp}) and (\ref{final})).
As a final application, we shall see in (\ref{m4}) how the pfaffian can be used to characterize the transpose involution on a split biquaternion algebra.

\section{Preliminaries}
Let $V$ be a finite dimensional vector space over a field $F$.
A {\it quadratic form} over $F$ is a map $q:V\rightarrow F$ such that
$(i)$ $q(av)=a^2q(v)$ for every $a\in F$ and $v\in V$;
$(ii)$ the map $\mathfrak{b}_q:V\times V\rightarrow F$ defined by $\mathfrak{b}_q(u,v)=q(u+v)-q(u)-q(v)$
is a bilinear form.
The map $\mathfrak{b}_q$ is called the {\it polar form} of $q$.
Note that for every $v\in V$ we have $\mathfrak{b}_q(v,v)=2q(v)$.
In particular, if $\car F=2$, then $\mathfrak{b}_q(v,v)=0$ for all $v\in V$, i.e., $\mathfrak{b}_q$ is an
{\it alternating} form.
The {\it orthogonal complement} of a subspace $W\subseteq V$ is defined as $W^\perp=\{x\in V\mid b_q(x,y)=0\ {\rm for\ all}\ y\in W\}$.

A quadratic form $q$ (resp. a bilinear form $\mathfrak{b}$) on $V$ is called {\it isotropic} if there exists a nonzero vector $v\in V$ such that $q(v)=0$ (resp. $\mathfrak{b}(v,v)=0$).
For $\alpha\in F$, we say that $q$ (resp. $\mathfrak{b}$) {\it represents} $\alpha$ if there exists a nonzero vector $v\in V$ such that $q(v)=\alpha$ (resp. $\mathfrak{b}(v,v)=\alpha$).
 The sets of all elements of $F$ represented by $q$ and $\mathfrak{b}$ are denoted by $D_F(q)$ and $D_F(\mathfrak{b})$
 respectively.
For $\alpha\in F^\times$, the {\it scaled} quadratic form $\alpha\cdot q$ is defined as $\alpha\cdot q(v)=\alpha q(v)$ for every $v\in V$.

For $a_1,\cdots,a_n\in F$, the isometry class of the quadratic form $\sum_{i=1}^na_ix_i^2$ is denoted by $\langle a_1,\cdots,a_n\rangle_q$.
Also, the isometry class of the bilinear form
$\sum_{i=1}^na_ix_iy_i$ is denoted by $\langle a_1,\cdots,a_n\rangle$.
Finally, the form $\lla a_1,\cdots,a_n\rra:=\langle 1,a_1\rangle\otimes\cdots\otimes\langle 1,a_n\rangle$ is called a {\it bilinear $n$-fold Pfister form}.

An {\it involution} on a central simple $F$-algebra $A$ is an antiautomorphism $\sigma$ of $A$ of order $2$.
We say that $\sigma$ is {\it  of the first kind} if $\sigma|_F=\id$.
An involution $\sigma$ of the first kind is said to be of {\it symplectic} if over a splitting field of $A$, it becomes adjoint to an alternating bilinear form.
Otherwise, $\sigma$ is called {\it orthogonal}.
The set of {\it alternating elements} of $A$ is defined as
$\alt(A,\sigma)=\{a-\sigma(a) \mid a\in A\}$.
If $A$ is of even degree $2m$, the {\it discriminant} of an orthogonal involution $\sigma$ on $A$ is defined as $\disc\sigma=(-1)^m\nrd_A(x) F^{\times2}\in F^\times/F^{\times2}$, where  $x\in\alt(A,\sigma)$ is a unit and $\nrd_A(x)$ is the reduced norm of $x$  in $A$.
Note that by \cite[(7.1)]{knus}, the discriminant does not depend on the choice of $x\in\alt(A,\sigma)$.

A {\it quaternion algebra} over a field $F$ is a central simple algebra $Q$ of degree $2$.
The {\it canonical} involution $\gamma$ on $Q$ is defined by $\gamma(x)=\trd_Q(x)-x$ for $x\in Q$, where $\trd_A(x)$ is the reduced trace of $x$  in $A$.
The canonical involution on $Q$ is the unique involution of symplectic type on $Q$ and it satisfies $\gamma(x)x\in F$ for every $x\in Q$ (see \cite[Ch. 2]{knus}).
The map $N_Q:Q\rightarrow F$ defined by $N_Q(x)=\gamma(x)x$ is called the {\it norm form} of $Q$.
An element $x\in Q$ is called a {\it pure quaternion} if $\trd_Q(x)=0$.
The set of all pure quaternions of $Q$ is a $3$-dimensional subspace of $Q$ denoted by $Q_0$.
Note that an element $x\in Q$ lies in $Q_0$ if and only if $\gamma(x)=-x$, or equivalently, $N_Q(x)=-x^2$.

A central simple $F$-algebra with involution $(A,\sigma)$ is called {\it totally decomposable} if it decomposes as a tensor product of $\sigma$-invariant quaternion $F$-algebras.
If $A$ is a biquaternion algebra, we will use the term {\it decomposable} instead of totally decomposable.
Note that a biquaternion algebra with orthogonal involution $(A,\sigma)$ is decomposable if and only if
$\disc\sigma$ is trivial (see \cite[(3.7)]{parimala}).

\section{The pfaffian and the pfaffian adjoint}\label{sec-pfaffian}
We begin our discussion by looking at special cases of \cite[(2.1)]{knus2} and \cite[(3.1)]{knus2}.
\begin{thm}\label{pfsigma}
Let $(A,\sigma)$ be a biquaternion algebra with orthogonal involution over a field $F$ and let $d_\sigma\in F^\times$ be a representative of the class $\disc\sigma\in F^\times/F^{\times2}$.
There exists a map $pf_\sigma:\alt(A,\sigma)\rightarrow F$ such that
$pf_\sigma(x)^2=d_\sigma{\nrd}_A(x)$ for every $x\in \alt(A,\sigma)$.
The map $pf_\sigma$ is uniquely determined up to a sign.
Moreover, there exists an $F$-linear map $\pi_\sigma:\alt(A,\sigma)\rightarrow \alt(A,\sigma)$ such that $x\pi_\sigma(x)=\pi_\sigma(x)x=pf_\sigma(x)$ and $\pi_\sigma^2(x)=d_\sigma x$ for every $x\in \alt(A,\sigma)$.
\end{thm}

\begin{rem}\label{inve}
The map $\pi_\sigma$ in (\ref{pfsigma}) is uniquely determined by $pf_\sigma$.
Indeed, it is easily seen by scalar extension to a splitting field that $\alt(A,\sigma)$ has a basis $\mathcal{B}$ consisting of invertible elements.
For every $x\in\mathcal{B}$, we must have $\pi_\sigma(x)\penalty 0 =x^{-1}pf_\sigma(x)$.
As $\pi_\sigma$ is $F$-linear, it is uniquely defined on $\alt(A,\sigma)$.
\end{rem}

\begin{defi}
A map $pf_\sigma$ as in (\ref{pfsigma}) is called a {\it pfaffian of} $(A,\sigma)$.
We also call the map $\pi_\sigma$, the {\it pfaffian adjoint}
 of $pf_\sigma$.
\end{defi}
Note that by \cite[(3.3)]{knus5}, every pfaffian of $(A,\sigma)$ is an Albert form of $A$.

\begin{notation}\label{not}
Let $(A,\sigma)$ be a decomposable biquaternion algebra with ortho\-gonal involution over a field $F$.
Since $\disc\sigma$ is trivial, by (\ref{pfsigma}) there is a unique, up to a sign, pfaffian $pf_\sigma$ satisfying
$pf_\sigma(x)^2={\nrd}_A(x)$ for $x\in\alt(A,\sigma)$.
We denote this pfaffian by $q_\sigma$.
We also denote by $p_\sigma$ the pfaffian adjoint of $q_\sigma$,~hence
\[q_\sigma(x)^2={\nrd}_A(x),\quad xp_\sigma(x)=p_\sigma(x)x=q_\sigma(x)\quad {\rm and} \quad p_\sigma^2(x)= x,\]
for every $x\in\alt(A,\sigma)$.
We also use the following notation:
\begin{align*}
\alt(A,\sigma)^+&:=\{x+p_\sigma(x)\mid x\in\alt(A,\sigma)\},\\
\alt(A,\sigma)^-&:=\{x-p_\sigma(x)\mid x\in\alt(A,\sigma)\}.
\end{align*}
\end{notation}
Note that if $\car F=2$, then $\alt(A,\sigma)^+=\alt(A,\sigma)^-$.
Also, as proved in \cite[p. 597]{knus5} and \cite[(3.5)]{parimala}, $\alt(A,\sigma)^+$ and $\alt(A,\sigma)^-$ are $3$-dimensional subspaces of $\alt(A,\sigma)$.
Since $p_\sigma^2=\id$, we have $p_\sigma(x)=x$
for every $x\in\alt(A,\sigma)^+$ and $p_\sigma(x)=-x$
for every $x\in\alt(A,\sigma)^-$.
The converse is also true, i.e.,
\begin{align}\label{eq5}
\alt(A,\sigma)^+&=\{x\in\alt(A,\sigma)\mid p_\sigma(x)=x\},\\
\alt(A,\sigma)^-&=\{x\in\alt(A,\sigma)\mid p_\sigma(x)=-x\}.\label{eq7}
\end{align}
Indeed, if $\car F\neq2$, then for every $x\in \alt(A,\sigma)$ with $p_\sigma(x)=x$ we have
$x=\frac{1}{2}(x+p_\sigma(x))\in\alt(A,\sigma)^+$.
Similarly if $p_\sigma(x)=-x$, then
$x=\frac{1}{2}(x-p_\sigma(x))\in\alt(A,\sigma)^-$.
If $\car F=2$, then the relation (\ref{eq5}) follows from the dimension formula for the image and the kernel of the linear map $p_\sigma+\id$.

The next result is implicitly contained in \cite[pp. 249-250]{knus}.
\begin{lem}\label{iso}
Let $(A,\sigma)$ be a decomposable biquaternion algebra with ortho\-gonal involution over a field $F$.
Then $p_\sigma$ is an isometry of $(\alt(A,\sigma),q_\sigma)$.
Furthermore, $\mathfrak{b}_{q_\sigma}(x,y)=xp_\sigma(y)+yp_\sigma(x)$, for $x,y\in\alt(A,\sigma)$.
\end{lem}

\begin{proof}
For every $x\in\alt(A,\sigma)$ we have
$q_\sigma(p_\sigma(x))=p_\sigma(p_\sigma(x))p_\sigma(x)=xp_\sigma(x)=q_\sigma(x)$.
Thus, $p_\sigma$ is an isometry.
The second assertion is easily obtained from the relations
$q_\sigma(x)=xp_\sigma(x)$ and $\mathfrak{b}_{q_\sigma}(x,y)=q_\sigma(x+y)-q_\sigma(x)-q_\sigma(y)$.
\end{proof}

\begin{lem}\label{rema}
Let $(A,\sigma)$ be a decomposable biquaternion algebra with ortho\-gonal involution over a field $F$.
Then $\alt(A,\sigma)^+=(\alt(A,\sigma)^-)^\perp\subseteq C_A(\alt(A,\sigma)^-)$.
\end{lem}
\begin{proof}
Let $\mathfrak{b}=\mathfrak{b}_{q_\sigma}$ and let $x\in\alt(A,\sigma)^+$.
By (\ref{iso}), $p_\sigma$ is an isometry of $(\alt(A,\sigma),q_\sigma)$, hence
$\mathfrak{b}(x,y)=\mathfrak{b}(p_\sigma(x),p_\sigma(y))=\mathfrak{b}(x,p_\sigma(y))$ for every $y\in\alt(A,\sigma)$.
Thus, $\mathfrak{b}(x,y-p_\sigma(y))=0$, i.e., $\alt(A,\sigma)^+\subseteq(\alt(A,\sigma)^-)^\perp$.
By dimension count we obtain $\alt(A,\sigma)^+=(\alt(A,\sigma)^-)^\perp$.
Now let $z\in\alt(A,\sigma)^-$.
By (\ref{iso}) we have $0=\mathfrak{b}(x,z)=-xz+zx$.
Thus, $xz=zx$, which implies that $\alt(A,\sigma)^+$ commutes with $\alt(A,\sigma)^-$, i.e., $\alt(A,\sigma)^+\subseteq C_A(\alt(A,\sigma)^-)$.
\end{proof}
\begin{lem}\label{psigma}
Let $(A,\sigma)$ be a decomposable biquaternion algebra with ortho\-gonal involution over a field $F$ and let $x\in\alt(A,\sigma)$.
If $x^2\in F$, then $p_\sigma(x)=\pm x$.
\end{lem}

\begin{proof}
Set $\alpha=x^2\in F$ and $\beta=q_\sigma(x)\in F$.
Then $\beta^2=q_\sigma(x)^2=\nrd_A(x)=\pm\alpha^2$.
Thus, $\beta=\lambda\alpha$ for some $\lambda\in F$ with $\lambda^4=1$, i.e., $q_\sigma(x)=\lambda x^2$.
If $\alpha\neq0$, then multiplying $xp_\sigma(x)=q_\sigma(x)=\lambda x^2$ on the left by $x^{-1}$ we obtain $p_\sigma(x)=\lambda x$.
The relation $p_\sigma^2=\id$ then implies that $\lambda=\pm1$ and we are done.
So suppose that $\alpha=0$, i.e., $x^2=0$.
By (\ref{iso}) we have $\mathfrak{b}_{q_\sigma}(p_\sigma(x),x)=p_\sigma(x)^2+x^2=p_\sigma(x)^2$, hence $p_\sigma(x)^2\in F$.
On the other hand, the relations $xp_\sigma(x)=q_\sigma(x)=\lambda x^2=0$ show that $p_\sigma(x)$ is not invertible.
Thus,
\begin{align}\label{eq11}
p_\sigma(x)^2=0.
\end{align}
Suppose that $p_\sigma(x)\neq x$, hence $x\notin\alt(A,\sigma)^+$.
In view of (\ref{rema}) one can find $w\in\alt(A,\sigma)^-$ such that $\mathfrak{b}_{q_\sigma}(x,w)=1$.
By (\ref{iso}) we have
\begin{align}\label{eq10}
-xw+wp_\sigma(x)=1.
\end{align}
Multiplying (\ref{eq10}) on the left by $x$ we get $xwp_\sigma(x)=x$.
Using (\ref{eq10}), it follows that $(wp_\sigma(x)-1)p_\sigma(x)=x$, which yields $p_\sigma(x)=-x$ by (\ref{eq11}).
This completes the proof
(note that if $\car F=2$, this argument shows that the assumption $p_\sigma(x)\neq x$ leads to the contradiction $p_\sigma(x)=-x$, hence $p_\sigma(x)=x$).
\end{proof}
The next result follows from (\ref{psigma}) and the relations (\ref{eq5}) and (\ref{eq7}) below (\ref{not}).
\begin{prop}\label{x2}
Let $(A,\sigma)$ be a decomposable biquaternion algebra with ortho\-gonal involution over a field $F$ and let $\alt(A,\sigma)^0=\alt(A,\sigma)^+\cup\alt(A,\sigma)^-$.
Then
$\alt(A,\sigma)^0=\{x\in\alt(A,\sigma)\mid p_\sigma(x)=\pm x\}=\{x\in\alt(A,\sigma)\mid x^2\in F\}$.
\end{prop}

\begin{notation}
For a decomposable biquaternion algebra with involution of orthogonal type $(A,\sigma)$ over a field $F$, we use the notation
$Q(A,\sigma)^+=F+\alt(A,\sigma)^+$ and $Q(A,\sigma)^-=F+\alt(A,\sigma)^-$.
We will simply denote $Q(A,\sigma)^+$ by $Q^+$ and $Q(A,\sigma)^-$ by $Q^-$,
if the pair $(A,\sigma)$ is clear from the context.
\end{notation}

\begin{lem}{\rm (\cite{parimala})}\label{q}
Let $(A,\sigma)$ be a decomposable biquaternion algebra with ortho\-gonal involution over a field $F$.
\begin{itemize}
\item[$(1)$] If $\car F\neq2$, then $Q^+$ and $Q^-$ are two $\sigma$-invariant quaternion subalgebras of $A$ with
$Q^+_0=\alt(A,\sigma)^+$ and $Q^-_0=\alt(A,\sigma)^-$.
Furthermore, we have $(A,\sigma)\simeq(Q^+,\sigma|_{Q^+})\otimes(Q^-,\sigma|_{Q^-})$, where
 $\sigma|_{Q^+}$ and $\sigma|_{Q^-}$ are the canonical involutions of $Q^+$ and $Q^-$ respectively.
\item[$(2)$] If $\car F=2$, then $Q^+=Q^-$ is a maximal commutative subalgebra of $F$ satisfying $x^2\in F$ for every $x\in Q^+$.
\end{itemize}
\end{lem}

\begin{proof}
Assume first that $\car F\neq2$.
As observed in \cite[(3.5)]{parimala}, $Q^+$ is a $\sigma$-invariant quaternion subalgebra of $A$ and $\sigma|_{Q^+}$ is of symplectic type.
By dimension count and (\ref{rema}) we obtain $Q^-= C_A(Q^+)$, hence $A\simeq Q^+\otimes_F Q^-$.
By \cite[(2.23 (1))]{knus}, $\sigma|_{Q^-}$ is of symplectic type.
Finally, since $\trd_{Q^+}(x)=0$ for every $x\in\alt(A,\sigma)^+$, we have $Q^+_0=\alt(A,\sigma)^+$.
Similarly $Q^-_0=\alt(A,\sigma)^-$.
This proves the first part.
The second part follows from \cite[(3.6)]{parimala}.
\end{proof}

\begin{notation}
Let $(A,\sigma)$ be a decomposable biquaternion algebra with ortho\-gonal involution over a field $F$.
We denote by $q^+_\sigma$ and $q^-_\sigma$ the restrictions of $q_\sigma$ to $\alt(A,\sigma)^+$ and $\alt(A,\sigma)^-$ respectively.
\end{notation}

\begin{lem}\label{bas}
Let $(A,\sigma)$ be a decomposable biquaternion algebra with orthogonal involution over a field $F$.
\begin{itemize}
\item[$(1)$] Every unit $u\in\alt(A,\sigma)^+$ {\rm(}resp. $u\in\alt(A,\sigma)^-${\rm)} can be extended to a basis $(u,v,w)$ of $\alt(A,\sigma)^+$ {\rm(}resp. $\alt(A,\sigma)^-${\rm)} such that $w=uv$.
\item[$(2)$] Every basis $(u,v,w)$ of $\alt(A,\sigma)^+$ {\rm(}resp. $\alt(A,\sigma)^-${\rm)} with $w=uv$ is ortho\-gonal with respect to the polar form of $q_\sigma^+$ {\rm(}resp. $q_\sigma^-${\rm)}.
\item[$(3)$] If $\car F\neq2$, then $N_{Q^+}\simeq\langle1\rangle_q\perp(-1)\cdot q_\sigma^+$ and $N_{Q^-}\simeq\langle1\rangle_q\perp q_\sigma^-$.
\item[$(4)$] If $\car F=2$ and $(A,\sigma)\simeq(Q_1,\sigma_1)\otimes(Q_2,\sigma_2)$ is a decomposition of $(A,\sigma)$, then $q_{\sigma}^+\simeq\langle\alpha,\beta,\alpha\beta\rangle_q$, where $\alpha\in F^\times$ and $\beta\in F^\times$ are representatives of the classes
$\disc\sigma_1\in F^\times/F^{\times2}$ and
$\disc\sigma_2\in F^\times/F^{\times2}$ respectively.
\end{itemize}

\end{lem}

\begin{proof}
We just prove the result for $q_\sigma^+$.
The proof for $q_\sigma^-$ is similar.

(1) Choose an element $u'\in\alt(A,\sigma)^+\setminus Fu$ and set $\alpha=u^2\in F^\times$.
By (\ref{q}), $uu'\in Q^+=F+\alt(A,\sigma)^+$.
Thus, there exist $\lambda\in F$ and $w\in\alt(A,\sigma)^+$ such that $uu'=\lambda+w$.
Set $v=u'-\lambda\alpha^{-1}u\in\alt(A,\sigma)^+$.
Then $uv=w\in\alt(A,\sigma)^+$.
Thus, $(u,v,w)$ is the desired basis.

(2) Let $\mathcal{B}=(u,v,w)$ be a basis of $\alt(A,\sigma)^+$ with $w=uv$.
Then $vu=\sigma(uv)=-uv$.
Using (\ref{iso}) we obtain $\mathfrak{b}(u,v)=uv+vu=0$, where $\mathfrak{b}$ is the polar form of $q_\sigma^+$.
Similarly, $\mathfrak{b}(u,w)=\mathfrak{b}(v,w)=0$.

(3) Let $(u,v,w)$ be a basis of $\alt(A,\sigma)^+$ with $w=uv$.
By (2), $q_{\sigma}^+\simeq\langle\alpha,\beta,-\alpha\beta\rangle_q$, where $\alpha=u^2\in F$ and $\beta=v^2\in F$.
Since $vu=-uv$, $(1,u,v,w)$ is a quaternion basis of $Q^+$.
Thus, $N_{Q^+}\simeq\langle1,-\alpha,-\beta,\alpha\beta\rangle_q$ by \cite[(9.6)]{elman}.

(4) Let $u\in\alt(Q_1,\sigma_1)$ and $v\in\alt(Q_2,\sigma_2)$ be two units and set $\alpha=u^2\in F^\times$, $\beta=v^2\in F^\times$ and $w=uv$.
By (\ref{x2}) we have $u,v\in\alt(A,\sigma)^+$.
Also, $\disc\sigma_1=\alpha F^{\times2}\in F^\times/F^{\times2}$ and
$\disc\sigma_2=\beta F^{\times2}\in F^\times/F^{\times2}$.
Since $w\in\alt(A,\sigma)$ and $w^2\in F$, by (\ref{x2}) we obtain $w\in\alt(A,\sigma)^+$.
So $(u,v,w)$ is a basis of $\alt(A,\sigma)^+$ and $q_{\sigma}^+\simeq\langle\alpha,\beta,\alpha\beta\rangle_q$.
\end{proof}

\begin{prop}\label{jay}{\rm (Compare \cite[(5.3)]{knus5})}
Let $(A,\sigma)$ and $(A',\sigma')$ be decomposable biquaternion algebras with orthogonal involution over a field $F$.
If $(A,\sigma)\simeq(A',\sigma')$, then either $q_\sigma\simeq q_{\sigma'}$ and $q_\sigma^+\simeq q_{\sigma'}^+$ or $q_\sigma\simeq (-1)\cdot q_{\sigma'}$ and $q_\sigma^+\simeq q_{\sigma'}^-$.
\end{prop}

\begin{proof}
Let $\phi:(A,\sigma)\iso(A',\sigma')$ be an isomorphism of $F$-algebras with involution.
Then $\phi(\alt(A,\sigma))=\alt(A',\sigma')$ and
\[q_{\sigma'}(\phi(x))^2={\nrd}_{A'}(\phi(x))={\nrd}_{A}(x)=q_\sigma(x)^2,\quad {\rm for}\ x\in\alt(A,\sigma).\]
Thus, $q_\sigma'\circ\phi=\pm q_{\sigma}$.
Suppose first that $q_\sigma'\circ\phi=q_{\sigma}$.
Then $\phi$ restricts to an isometry $f:(\alt(A,\sigma),q_\sigma)\rightarrow(\alt(A',\sigma'),q_{\sigma'})$.
Set $h=f\circ p_\sigma\circ f^{-1}$.
Then $h$ is an endomorphism of $\alt(A',\sigma')$.
We claim that $h=p_{\sigma'}$.
For every $x\in\alt(A',\sigma')$ we have $h^2(x)=f\circ p_\sigma^2\circ f^{-1}(x)=f\circ\id\circ f^{-1}(x)=x$ and
\begin{align*}
xh(x)&=xf(p_\sigma( f^{-1}(x)))=x\phi(p_\sigma( f^{-1}(x)))=\phi( f^{-1}(x))\phi(p_\sigma( f^{-1}(x)))\\
&=\phi(f^{-1}(x)p_\sigma( f^{-1}(x)))=\phi(q_\sigma( f^{-1}(x)))=\phi(q_{\sigma'}(x))=q_{\sigma'}(x).
\end{align*}
Similarly, we have $h(x)x=q_{\sigma'}(x)$ for every $x\in\alt(A',\sigma')$.
Thus, $h=p_{\sigma'}$ and the claim is proved.
It follows that $p_{\sigma'}\circ f= f\circ p_\sigma$.
Now, if $x\in\alt(A,\sigma)^+$, then $p_\sigma(x)=x$,
 which yields $p_{\sigma'}( f(x))=f(p_\sigma(x))= f(x)$.
It follows that $ f(x)\in\alt(A',\sigma')^+$, i.e., $f$ restricts to an isometry $q_\sigma^+\simeq q_{\sigma'}^+$.
A similar argument shows that if $q_\sigma'\circ\phi=-q_{\sigma}$, then $q_\sigma^+\simeq q_{\sigma'}^-$.
\end{proof}

The next result complements \cite[(5.3)]{knus5} for biquaternion algebras.
\begin{thm}\label{charn}
Let $(A,\sigma)$ and $(A',\sigma')$ be two decomposable biquaternion algebras with orthogonal involution over a field $F$ of characteristic different from $2$.
Let $Q^+=Q(A,\sigma)^+$, $Q^-=Q(A,\sigma)^-$, ${Q'}^+=Q(A',\sigma')^+$ and ${Q'}^-=Q(A',\sigma')^-$.
The following statements are equivalent.
\begin{itemize}
\item[$(1)$] $(A,\sigma)\simeq(A',\sigma')$.
\item[$(2)$] Either $q_\sigma\simeq q_{\sigma'}$ and $q_\sigma^+\simeq q_{\sigma'}^+$ or $q_\sigma\simeq (-1)\cdot q_{\sigma'}$ and $q_\sigma^+\simeq q_{\sigma'}^-$.
\item[$(3)$] $A\simeq A'$ and either $q_\sigma^+\simeq q_{\sigma'}^+$ or $q_\sigma^+\simeq q_{\sigma'}^-$.
\item[$(4)$] $A\simeq A'$ and either $Q^+\simeq {Q'}^+$ or $Q^+\simeq Q^-$.
\end{itemize}
\end{thm}

\begin{proof}
The implication $(1)\Rightarrow(2)$ follows from (\ref{jay}).
Since $q_\sigma$ and $q_{\sigma'}$ are Albert forms of $(A,\sigma)$ and $(A',\sigma')$ respectively,
the condition $q_\sigma\simeq q_{\sigma'}$ (resp. $q_\sigma\simeq (-1)\cdot q_{\sigma'}$) implies that $A\simeq A'$, proving
$(2)\Rightarrow(3)$.
The implication $(3)\Rightarrow(4)$ follows from
 (\ref{bas} (3)) and \cite[Ch. III, (2.5)]{lam}.
To prove $(4)\Rightarrow(1)$ assume first $Q^+\simeq {Q'}^+$.
By (\ref{q} (1)) we have $C_A(Q^+)=Q^-$ and $C_{A'}({Q'}^+)={Q'}^-$.
Thus, the isomorphisms $Q^+\simeq_F {Q'}^+$ and $A\simeq_FA'$ imply that $Q^-\simeq_F {Q'}^-$.
Since the restrictions of $\sigma$ to $Q^+$ and ${Q}^-$ and the restrictions of $\sigma'$ to ${Q'}^+$ and ${Q'}^-$ are all symplectic, we obtain
\begin{align*}
(A,\sigma)&\simeq_F(Q^+,\sigma|_{Q^+})\otimes_F(Q^-,\sigma|_{Q^-})\\
&\simeq_F({Q'}^+,\sigma'|_{{Q'}^+})\otimes_F({Q'}^-,\sigma'|_{{Q'}^-})\simeq_F(A',\sigma').
\end{align*}
A similar argument works if $Q^+\simeq {Q'}^-$.
\end{proof}

\section{Relation with the Pfister invariant in characteristic two}\label{sec-pfister}
Throughout this section, $F$ is a field of characteristic $2$.
\begin{defi}
Let $A$ be a finite-dimensional associative $F$-algebra.
The minimum number $r$ such that $A$ can be generated as an $F$-algebra
by $r$ elements is called the {\it minimum rank} of $A$ and is denoted by $r_F(A)$.
\end{defi}

\begin{thm}\label{mah}{\rm (\cite{mahmoudi2})}
Let $(A,\sigma)$ be a totally decomposable algebra with invo\-lution of orthogonal type over $F$.
There exists a symmetric and self-centralizing
subalgebra $S\subseteq A$ such that $x^2\in F$ for every $x\in S$ and
$\dim_F S\penalty 0 =2^n$, where $n=r_F(S)$.
Furthermore, for every subalgebra $S$ with these properties, we have $S=F+S_0$, where $S_0=S\cap\alt(A,\sigma)$.
In particular, $S\subseteq F+\alt(A,\sigma)$.
Finally, the subalgebra $S$ is uniquely determined up to isomorphism.
\end{thm}

\begin{proof}
See \cite[(4.6) and (5.10)]{mahmoudi2}.
\end{proof}

\begin{notation}
We denote the algebra $S$ in (\ref{mah}) by $\Phi(A,\sigma)$.
\end{notation}

 The next result shows that for biquaternion algebras with orthogonal involution, the subalgebra $\Phi(A,\sigma)$ is unique as a set.
\begin{cor}\label{phi}
Let $(A,\sigma)$ be a decomposable biquaternion algebra with involution of orthogonal type over $F$.
Then $\Phi(A,\sigma)=Q^+$.
\end{cor}

\begin{proof}
Write $\Phi(A,\sigma)=F+S_0$, where $S_0=\Phi(A,\sigma)\cap\alt(A,\sigma)$.
Since every element of $\Phi(A,\sigma)$ is square-central, using (\ref{x2}) we have $S_0\subseteq\alt(A,\sigma)^+$.
Then $S_0=\alt(A,\sigma)^+$ by dimension count,
hence $\Phi(A,\sigma)=F+\alt(A,\sigma)^+=Q^+$.
\end{proof}

\begin{lem}\label{setalt}
Let $(A,\sigma)$ be a totally decomposable algebra of degree $2^n$ with orthogonal involution over $F$.
If there exists a set $\{u_1,\cdots,u_n\}\subseteq\alt(A,\sigma)$ consisting of pairwise commutative square-central units such that
$u_{i_1}\cdots u_{i_l}\in\alt(A,\sigma)$ for every $1\leq l\leq n$ and $1\leq i_1<\cdots<i_l\leq n$,
then $\Phi(A,\sigma)\simeq F[u_1,\cdots,u_n]$.
\end{lem}

\begin{proof}
By \cite[(2.2.3)]{jacob}, $S:=F[u_1,\cdots,u_n]$ is self-centralizing.
The other required properties of $S$, stated in (\ref{mah}), are easily verified.
\end{proof}

\begin{defi}
A set $\{u_1,\cdots,u_n\}\subseteq\alt(A,\sigma)$ as in (\ref{setalt}) is called a {\it set of alternating generators} of $\Phi(A,\sigma)$.
\end{defi}

We recall the following definition from \cite{dolphin3}.
\begin{defi}
Let $(A,\sigma)=(Q_1,\sigma_1)\otimes\cdots\otimes(Q_n,\sigma_n)$ be a totally decomposable algebra with orthogonal involution over $F$.
Let $\alpha_i\in F^\times$, $i=1,\cdots,n$, be a representative of the class $\disc\sigma_i\in F^\times/F^{\times2}$.
The bilinear $n$-fold Pfister form $\lla\alpha_1,\cdots,\alpha_n\rra$ is called the {\it Pfister invariant} of $(A,\sigma)$ and is denoted by $\mathfrak{Pf}(A,\sigma)$.
\end{defi}
Note that by \cite[(7.5)]{dolphin3}, $\mathfrak{Pf}(A,\sigma)$ is independent of the decomposition of $(A,\sigma)$.
Also, as observed in \cite[pp. 223-224]{mahmoudi2}, $\mathfrak{Pf}(A,\sigma)\simeq\lla\alpha_1,\cdots,\alpha_n\rra$ if and only if there exists a set of alternating generators $\{u_1,\cdots,u_n\}$ of $\Phi(A,\sigma)$ such that $u_i^2=\alpha_i\in F^\times$, $i=1,\cdots,n$.

\begin{lem}\label{repl}
Let $\lla\alpha,\beta\rra$ be an isotropic bilinear Pfister form over $F$.
If $\alpha\beta\neq0$, then $\lla\alpha,\beta\rra\simeq\lla\alpha,\beta+\alpha^{-1}\lambda^2\rra$ for every $\lambda\in F$.
\end{lem}

\begin{proof}
Since $\lla\alpha,\beta\rra$ is isotropic, by \cite[(4.14)]{elman} either $\alpha\in F^{\times2}$ or $\beta\in D_F\langle1,\alpha\rangle$.
If $\alpha\in F^{\times2}$,
 using \cite[(4.15 (2))]{elman} and \cite[(4.15 (1))]{elman} we obtain
\begin{align*}
\lla\alpha,\beta\rra&\simeq\lla\beta+\alpha^{-1}\lambda^2,\alpha\beta\rra
\simeq\lla\beta+\alpha^{-1}\lambda^2,\alpha\beta(\alpha^{-1}\lambda^2-(\beta+\alpha^{-1}\lambda^2))\rra\\
&\simeq\lla\beta+\alpha^{-1}\lambda^2,\alpha\beta^2\rra\simeq\lla\alpha,\beta+\alpha^{-1}\lambda^2\rra.
\end{align*}
If $\beta\in D_F\langle1,\alpha\rangle$, then there exist $b,c\in F$ such that $\beta=b^2+c^2\alpha$.
Let $s=\alpha^{-1}\beta^{-1}\lambda\in F$.
Using \cite[(4.15 (1))]{elman} we obtain
\begin{align*}
\lla\alpha,\beta\rra&\simeq\lla\alpha,\beta((1+cs\alpha)^2-(bs)^2\alpha)\rra
\simeq\lla\alpha,\beta(1+c^2s^2\alpha^2+b^2s^2\alpha)\rra\\
&\simeq\lla\alpha,\beta+s^2\alpha\beta(c^2\alpha+b^2)\rra\simeq\lla\alpha,\beta+s^2\alpha\beta^2\rra\simeq
\lla\alpha,\beta+\alpha^{-1}\lambda^2\rra.\quad \qedhere
\end{align*}
\end{proof}

\begin{lem}\label{pfq}
Let $(A,\sigma)$ be a decomposable biquaternion algebra with involution of orthogonal type over $F$ and let $\alpha,\beta\in F^\times$.
Then $\mathfrak{Pf}(A,\sigma)\simeq\lla\alpha,\beta\rra$ if and only if $q^+_\sigma\simeq\langle\alpha,\beta,\alpha\beta\rangle_q$.
\end{lem}

\begin{proof}
If $\mathfrak{Pf}(A,\sigma)\simeq\lla\alpha,\beta\rra$, then there exists a set of alternating generators $\{u,v\}$ of $\Phi(A,\sigma)$ such that $u^2=\alpha$ and $v^2=\beta$.
By (\ref{phi}) and (\ref{bas} (2)), $(u,v,uv)$ is an orthogonal basis of $\alt(A,\sigma)^+$, hence $q_\sigma^+\simeq\langle\alpha,\beta,\alpha\beta\rangle_q$.

To prove the converse,
choose a basis $(x,y,z)$ of $\alt(A,\sigma)^+$ with $x^2=\alpha$, $y^2=\beta$ and $z^2=\alpha\beta$.
Consider the element $xy\in \Phi(A,\sigma)$.
By (\ref{phi}), $\Phi(A,\sigma)=F+\alt(A,\sigma)^+$.
Thus, there exist $a,b,c,d\in F$ such that
\begin{align}\label{eq2}
xy=a+bx+cy+dz.
\end{align}
If $a=0$ then $xy=bx+cy+dz\in \alt(A,\sigma)^+$, which implies that $\{x,y\}$ is a set of alternating generators of $\Phi(A,\sigma)$.
As $x^2=\alpha$ and $y^2=\beta$ we obtain $\mathfrak{Pf}(A,\sigma)\simeq\lla\alpha,\beta\rra$.
So suppose that $a\neq 0$.
By squaring both sides of (\ref{eq2}), we obtain $\alpha\beta=a^2+b^2\alpha+c^2\beta+d^2\alpha\beta$,
which yields
\[1+(ba^{-1})^2\alpha+(ca^{-1})^2\beta+((d+1)a^{-1})^2\alpha\beta=0.\]
Therefore, the form $\lla\alpha,\beta\rra$ is isotropic.
Set $y'=y+\alpha^{-1}ax\in\alt(A,\sigma)^+$.
By (\ref{eq2}) we have $xy'=xy+a=bx+cy+dz\in\alt(A,\sigma)^+$, hence $\{x,y'\}$ is a set of alternating generators of $\Phi(A,\sigma)$.
As $x^2=\alpha$ and ${y'}^2=\beta+\alpha^{-1}a^2$, we obtain $\mathfrak{Pf}(A,\sigma)\simeq\lla\alpha,\beta+\alpha^{-1}a^2\rra$.
Thus, $\mathfrak{Pf}(A,\sigma)\simeq\lla\alpha,\beta\rra$ by (\ref{repl}).
\end{proof}

Using (\ref{pfq}) and (\ref{bas} (4)), we obtain the following relation between the Pfister invariant and the quadratic form $q_\sigma^+$.
\begin{prop}\label{pfister}
Let $(A,\sigma)$ and $(A',\sigma')$ be decomposable biquaternion algebras with orthogonal involution over $F$.
Then $q^+_\sigma\simeq q^+_{\sigma'}$ if and only if $\mathfrak{Pf}(A,\sigma)\simeq\mathfrak{Pf}(A',\sigma')$.
\end{prop}
The following result is analogous to (\ref{charn}).
\begin{thm}\label{char}
Let $(A,\sigma)$ and $(A',\sigma')$ be decomposable biquaternion algebras with orthogonal involution over $F$.
Then the following statements are equivalent.
\begin{itemize}
\item[$(1)$] $(A,\sigma)\simeq(A',\sigma')$.
\item[$(2)$] $q_\sigma\simeq q_{\sigma'}$ and $q^+_\sigma\simeq q^+_{\sigma'}$.
\item[$(3)$] $A\simeq A'$ and $q^+_\sigma\simeq q^+_{\sigma'}$.
\item[$(4)$] $A\simeq A'$ and $\mathfrak{Pf}(A,\sigma)\simeq\mathfrak{Pf}(A',\sigma')$.
\end{itemize}
\end{thm}

\begin{proof}
The implications
$(1)\Rightarrow (2)$ follows from (\ref{jay}).

$(2)\Rightarrow (3)$:
Since $q_\sigma$ and $q_{\sigma'}$ are Albert forms of $(A,\sigma)$ and $(A',\sigma')$ respectively,
$q_\sigma\simeq q_{\sigma'}$ implies that $A\simeq A'$ by \cite[(16.3)]{knus}.

The implications $(3)\Rightarrow (4)$ and $(4)\Rightarrow (1)$ follow from (\ref{pfister}) and \cite[(6.5)]{mahmoudi2} respectively.
\end{proof}

\begin{lem}\label{lem}
If $\lla\alpha,\beta\rra$ is an anisotropic bilinear Pfister form over $F$, then $\lla\alpha,\beta\rra\not\simeq\lla\alpha+1,\beta\rra$.
\end{lem}

\begin{proof}
As proved in \cite[p. 16]{arason}, two bilinear Pfister forms are isometric if and only if their pure subforms are isometric.
Thus, it is enough to show that the pure subform of $\lla\alpha,\beta\rra$ does not represent $\alpha+1$.
If $\alpha+1\in D_F(\langle\alpha,\beta,\alpha\beta\rangle)$, then there exist $a,b,c\in F$ such that
$a^2\alpha+b^2\beta+c^2\alpha\beta=\alpha+1$.
Thus, $1+(a+1)^2\alpha+b^2\beta+c^2\alpha\beta=0$, i.e., $\lla\alpha,\beta\rra$ is isotropic which contradicts the assumption.
\end{proof}

\begin{defi}
For $\alpha\in F^\times$, define an involution $T_\alpha:M_2(F)\rightarrow M_2(F)$ via
{\setlength\arraycolsep{2pt}
\begin{eqnarray*}
T_\alpha\left(\begin{array}{cc}a & b \\c & d\end{array}\right)=\left(\begin{array}{cc}a & c\alpha^{-1} \\b\alpha & d\end{array}\right).
\end{eqnarray*}}
\end{defi}
\noindent Note that $T_\alpha$ is of orthogonal type and $\disc T_\alpha=\alpha F^{\times2}\in F^\times/F^{\times2}$.

The following example shows that if $\car F=2$, the conditions $A\simeq_F A'$ and $Q^+\simeq_F {Q'}^+$ don't necessarily imply that $(A,\sigma)\simeq(A',\sigma')$ (compare (\ref{charn})).
\begin{ex}
Let $\lla\alpha,\beta\rra$ be an anisotropic Pfister form over a field $F$ of characteristic $2$ and let $A=M_4(F)$.
Consider the involutions $\sigma=T_\alpha\otimes T_\beta$ and $\sigma'=T_{\alpha+1}\otimes T_\beta$ on $A$.
Then $\mathfrak{Pf}(A,\sigma)\simeq\lla\alpha,\beta\rra$ and $\mathfrak{Pf}(A,\sigma')\simeq\lla\alpha+1,\beta\rra$, hence $\mathfrak{Pf}(A,\sigma)\not\simeq\mathfrak{Pf}(A,\sigma')$ by (\ref{lem}).
Using (\ref{char}), we obtain $(A,\sigma)\not\simeq(A,\sigma')$.

On the other hand there exists a set of alternating generators $\{u,v\}$ (resp. $\{u',v'\}$) of $\Phi(A,\sigma)$ (resp. $\Phi(A,\sigma')$) such that $u^2=\alpha$ and $v^2=\beta$ (resp. ${u'}^2=\alpha+1$ and ${v'}^2=\beta$).
Then $\Phi(A,\sigma)\simeq F[u,v]$ and $\Phi(A,\sigma')\simeq F[u',v']$.
The linear map $f:F[u,v]\rightarrow F[u',v']$ induced by $f(1)=1$, $f(u)=u'+1$, $f(v)=v'$ and $f(uv)=(u'+1)v'$ is an $F$-algebra isomorphism.
Thus, $\Phi(A,\sigma)\simeq\Phi(A,\sigma')$, which implies that $Q(A,\sigma)^+\simeq Q(A,\sigma')^+$ by (\ref{phi}).
\end{ex}

\section{Metabolic involutions}\label{sec-met}
Let $(A,\sigma)$ be an algebra with involution over a field $F$ of arbitrary characteristic.
An idempotent $e\in A$ is called {\it hyperbolic} (resp. {\it metabolic}) with respect to $\sigma$ if $\sigma(e)=1-e$ (resp. $\sigma(e)e=0$ and $(1-e)(1-\sigma(e))=0$).
The pair $(A,\sigma)$ is called {\it hyperbolic} (resp. {\it metabolic}) if $A$ contains a hyperbolic (resp. metabolic) idempotent with respect to $\sigma$.
Every hyperbolic involution $\sigma$ is metabolic but the converse is not always true.
If $\sigma$ is symplectic or $\car F\neq2$, the involution $\sigma$ is metabolic if and only if it is hyperbolic, (see \cite[(4.10)]{dolphin} and \cite[(A.3)]{fluckiger}).
\begin{lem}\label{metab}
Let $(A,\sigma)$ be a central simple algebra with orthogonal involution over a field $F$.
If $e\in A$ is a metabolic idempotent, then $(e-\sigma(e))^2=1$.
\end{lem}

\begin{proof}
This follows from the relations $(1-e)(1-\sigma(e))=0$ and $\sigma(e)e=0$.
\end{proof}

\begin{thm}\label{hyp}
Let $(A,\sigma)$ be a decomposable biquaternion algebra with ortho\-gonal involution over a field $F$.
The following statements are equivalent.
\begin{itemize}
\item[$(1)$] $(A,\sigma)$ is metabolic.
\item[$(2)$] $Q^+$ or $Q^-$ splits.
\item[$(3)$] $1\in D_F(q_\sigma^+)$ or $-1\in D_F(q_\sigma^-)$.
\item[$(4)$] $q_\sigma^+$ or $q_\sigma^-$ is isotropic.
\end{itemize}
\end{thm}

\begin{proof}
If $\car F\neq2$, by (\ref{q} (1)) we have $(A,\sigma)\simeq(Q^+,\sigma|_{Q^+})\otimes(Q^-,\sigma|_{Q^-})$, where $\sigma|_{Q^+}$ and $\sigma|_{Q^-}$ are the canonical involutions of $Q^+$ and $Q^-$ respectively.
Thus, the equivalence $(1)\Leftrightarrow(2)$ follows from \cite[(3.1)]{haile}.
The equivalences $(2)\Leftrightarrow(3)$ and $(2)\Leftrightarrow(4)$ both follow from (\ref{bas} (3)) and \cite[Ch. III, (2.7)]{lam}.

Now, let $\car F=2$.
Then the equivalence $(1)\Leftrightarrow(2)$ follows from \cite[(6.6)]{mahmoudi2}.

$(1)\Rightarrow(3)$:
Let $e$ be a metabolic idempotent with respect to $\sigma$ and let $x=e-\sigma(e)$.
By (\ref{metab}), we have $x^2=1$.
Since $x\in\alt(A,\sigma)$, (\ref{x2}) implies that $x\in\alt(A,\sigma)^+$, hence $q_\sigma^+(x)=1$.

$(3)\Rightarrow(4)$:
Suppose that $q_\sigma^+(u)=1$ for some $u\in\alt(A,\sigma)^+$.
By (\ref{bas} (1)) and (\ref{bas} (2)), the element $u$ extends to an orthogonal basis $(u,v,w)$ of $\alt(A,\sigma)^+$ with $w=uv$.
According to (\ref{q} (2)), $Q^+$ is commutative.
 Thus, $q_\sigma^+(v+w)=(v+w)^2=v^2+(uv)^2=0$, i.e., $q_\sigma^+$ is isotropic.

$(4)\Rightarrow(2)$:
If $q_\sigma^+$ is isotropic, then there exists a nonzero $x\in\alt(A,\sigma)^+\subseteq Q^+$ such that $x^2=0$, hence
 $Q^+$ splits.
\end{proof}

\begin{cor}\label{bayer}
Let $(A,\sigma)$ be a central simple algebra with involution over a field $F$.
If $\sigma$ is metabolic, then $\disc\sigma$ is trivial.
\end{cor}

\begin{proof}
The result follows from (\ref{metab}) if $\car F=2$ and
\cite[(2.3)]{fluckiger} if $\car F\neq2$.
\end{proof}

\begin{prop}\label{final}
Let $(A,\sigma)$ be a biquaternion algebra with involution of ortho\-gonal type over a field $F$.
Then $\sigma$ is metabolic if and only if there exists $u\in\alt(A,\sigma)$ such that $u^2=1$.
\end{prop}

\begin{proof}
If $\sigma$ is metabolic, then by (\ref{bayer}), $\disc\sigma$ is trivial.
Thus, $\sigma$ is decomposable and the result follows from (\ref{hyp}).
Conversely, suppose that there exists $u\in\alt(A,\sigma)$ such that $u^2=1$.
Then $\disc\sigma=\nrd_A(u)F^{\times2}$ is trivial, so $(A,\sigma)$ is decomposable by \cite[(3.7)]{parimala}.
Since $u^2=1\in F$ and $u\in\alt(A,\sigma)$, by (\ref{x2}) we have $u\in\alt(A,\sigma)^+\cup\alt(A,\sigma)^-$.
Therefore, either $u\in\alt(A,\sigma)^+$ (i.e., $q_\sigma^+(u)=1$) or $u\in\alt(A,\sigma)^-$ (i.e., $q_\sigma^-(u)=-1$).
By (\ref{hyp}), $\sigma$ is metabolic.
\end{proof}

\begin{prop}\label{m4}
Let $(A,\sigma)$ be a decomposable biquaternion algebra with ortho\-gonal involution over a field $F$.
Then $(A,\sigma)\simeq(M_4(F),t)$ if and only if $q_\sigma^+\simeq\langle-1,-1,-1\rangle_q$ and $q_\sigma^-\simeq\langle1,1,1\rangle_q$.
\end{prop}

\begin{proof}
If $\car F=2$, the result follows from \cite[(5.7)]{mahmoudi2} and (\ref{pfq}).
Suppose that $\car F\neq2$.
As observed in \cite[p. 235]{jacob}, $Q(M_4(F),t)^+$ has an $F$-basis $(1,u,v,w)$ subject to the relations
$u^2=-1$, $v^2=-1$ and $w=uv=-vu$.
By (\ref{bas} (2)) we obtain $q_t^+\simeq\langle-1,-1,-1\rangle_q$.
A similar argument shows that $q_t^-\simeq\langle1,1,1\rangle_q$.
Thus, the result follows from (\ref{charn}).
\end{proof}

\scriptsize
A.-H. Nokhodkar, {\tt
    a.nokhodkar@kashanu.ac.ir}\\
Department of Pure Mathematics, Faculty of Science, University of Kashan, P.~O. Box 87317-53153, Kashan, Iran.
\end{document}